\documentclass[12pt]{article}
\usepackage{amssymb,amsfonts,amsmath,amsthm,mathrsfs,enumitem}
\usepackage{mathtools}

\newtheorem{theorem}{Theorem}[section]
\newtheorem{lemma}[theorem]{Lemma}
\newtheorem{Corollary}[theorem]{Corollary}
\newtheorem{prop}[theorem]{Proposition}
\newtheorem{obs}[theorem]{Remark}
\newtheorem{definition}{Definition}[section]

\usepackage{graphicx}
\usepackage{color}
\usepackage[colorlinks = true,
linkcolor = blue,
urlcolor  = blue,
citecolor = blue,
anchorcolor = blue]{hyperref}

\numberwithin{equation}{section}

\begin{document}
	\title{\bf Classical Sobolev approach for a critical fourth-order Leray-Lions type problem: existence and multiplicity of solutions}
	
	\author{
		\small{{\bf Eduardo. H. Gomes Tavares}\thanks{Email: eduardogomes7107@gmail.com.}}\\
		{\small School of Computer Science and Technology, Dongguan University of Technology,} \\
		{\small 523808, Dongguan,  China.}\medskip\\
		\small {{\bf Angelo Guimarães}}\thanks{ Email: angelo.guimaraes@ueg.br.}\\
		{\small Academic Institute of Sciences and Technology, State University of Goiás,}\\
		{\small 76300-000, Ceres, Brazil.}\medskip\\
		\small {{\bf Edcarlos Domingos da Silva}}\thanks{ Email: edcarlos@ufg.br.}\\
		{\small Institute of Mathematics and Statistics, Federal University of Goias}\\
		{\small 74001-970, Goiania, Brazil.}\medskip\\
		\small {\bf Jin-Yun Yuan}  \footnote{Email: yuanjy@gmail.com (corresponding author).}\\
		{\small School of Computer Science and Technology, Dongguan University of Technology,}\\
		{\small 523808, Dongguan,  China.}}
	
		
		
		
	\date{}
	\maketitle
	
	\begin{abstract}  
		
		\noindent A fourth-order elliptic problem of Leray–Lions type is considered for combined nonlinearities and Sobolev-critical growth with Navier and Dirichlet boundary conditions. By combining variational methods and critical point theory, the existence and multiplicity of weak solutions are established in the setting of classical Sobolev spaces. Two distinct asymptotic regimes are considered for the perturbation term: sublinear and superlinear. In the sublinear case, the existence of infinitely many solutions is proved by using topological tools such as Krasnosel’skii’s genus and Clark’s deformation lemma. In the superlinear case, the existence of at least one nontrivial solution is obtained via the Mountain Pass Theorem. Furthermore, the applicability of the main results is illustrated in the context of Hamiltonian systems.
	\end{abstract}
	
	\smallskip
	

		\section{Introduction}
		
		Many physical phenomena, particularly those of a static nature, are governed by elliptic differential equations. A classical example arises in elasticity theory: under suitable assumptions, small vertical deflections of a thin plate subjected to a transverse load can be described by the Kirchhoff–Love equation \cite{Timoshenko-Woio}:  
		\begin{equation}\label{viga}  
			D\Delta^2u(x) = -q(x),  \quad x \in \Omega,
		\end{equation} 
		where $\Delta^2=\Delta\circ\Delta$ is the biharmonic operator, $D$ denotes the flexural rigidity of the plate, $\Omega\subset \mathbb{R}^N$ is a reference domain, $u=u(x)$ represents the vertical displacement at a point $x \in \Omega$, and $q=q(x)$ denotes the external load distribution. 
		
		It is evident that solutions of \eqref{viga} can be explicitly computed through successive integration. However, for problems involving more complex operators, it is not easy to establish the existence of solutions. To illustrate this point, consider the fourth-order {\it Leray–Lions} type problem \cite{Leray1965}:  
		\begin{equation}\label{viga2}  
			\Delta(f(x,\Delta u))=Q(x,u), \quad x \in \Omega,  
		\end{equation}  
		where $f, Q:\Omega\times\mathbb{R}\to \mathbb{R}$ satisfy suitable regularity and growth conditions. In this framework, \eqref{viga2} may encompass \eqref{viga} with sources \cite{EdmundsFortunatoJannelli} and also several problems with nonlinear operators, such as $p$-biharmonic equations \cite{wang2008nonuniformly}, weighted $p$-biharmonic equations \cite{Dridi2025}, $p(x)$-biharmonic equations \cite{Kefi2018} and $\phi$-biharmonic equations \cite{do2025hamiltonian}. Besides, \eqref{viga2} also serves as a bridge for conducting a rigorous qualitative analysis of Hamiltonian systems, as can be seen in \cite{bonheure2014hamiltonian, doO2003,do2025hamiltonian,fang2016nontrivial, guimaraes2023hamiltonian, lions}  and references therein. Moreover, due to its versatility,  \eqref{viga2} can be approached within different functional frameworks such as Orlicz–Sobolev spaces \cite{do2025hamiltonian}, Sobolev spaces with variable exponent \cite{MR3842319,Chung02112022} and classical Sobolev spaces \cite{Santos,drabek2001global, el2002spectrum,guimaraes2023hamiltonian}, where the latter framework corresponds to the original approach introduced by Leray and Lions in \cite{Leray1965}. For this reason, it constitutes the direction that will be explored in the present work.

		\subsection{State of the art}

		Regarding the model in \eqref{viga2}, we highlight several noteworthy results concerning the existence of weak and regular solutions in classical Sobolev spaces.
		
		In \cite{bisci2014multiple}, for
		\[
		f(x,t) = t^p, \quad p > 0,
		\]
		the existence of at least two non-trivial solutions was established for \eqref{viga2} in cases where the growth rate of \( Q(\cdot, t) \) varies from sublinear to linear in the subcritical regime. In \cite{ji2012p},  the existence of two distinct positive solutions was obtained for \eqref{viga2} via variational methods under the assumptions
		\[
		f(x,t) = t^p, \qquad Q(x,t)=\mu g(x)|t|^{s-1}t + |t|^{q-1}t
		\]
		where $\mu>0$, $g$ is a smooth function and $1<s<p<q$, implying that $Q(x,t)$ exhibits a concave–convex structure. 
		
		In the critical setting, the analysis becomes more delicate due to the lack of compactness inherent to Sobolev embedding. To overcome these challenges, some strategies were developed inspired by works involving second-order operators such as \cite{ambrosetti1994combined} for concave-convex nonlinearities and \cite{B-N} for sources with superlinear growth. This is initially evidenced in the work \cite{EdmundsFortunatoJannelli}, where sufficient conditions was proposed for the existence of nontrivial solutions to \eqref{viga2} in the critical case when
		\[
		f(x,t) = t, \qquad Q(x,t)=|t|^{8/(N-4)}t + \mu t, \quad N>4,
		\]
		where $\mu>0$. The quasilinear extension involving the \(p\)-biharmonic operator was later investigated in \cite{ederson}. More precisely, considering 
		\[
		f(x,t) = |t|^{1/(p-1)}t, \qquad Q(x,t)=\mu g(x)+ |t|^{q-1}t,
		\]
		where $\mu>0$, $g$ is a smooth function and $p,\, q$ satisfies
		\begin{equation}\label{pqHC}
			p>0, \quad q>0, \quad \frac{1}{p+1} + \frac{1}{q+1}= \frac{N-2}{N},
		\end{equation}
the Nehari manifold method and the Lions’ Concentration-Compactness Principle were combined to analyze the existence, non-existence, and regularity of solutions for \eqref{viga2}.  
		
		As previously mentioned, problem \eqref{viga2} can also appears as an intermediary tool in the study of Hamiltonian systems. In this context, the authors in \cite{EdJe-djairo} investigated the existence and multiplicity of solutions for a unilateral Hamiltonian system with critical growth perturbed by a superlinear subcritical term. In this setting, the corresponding fourth-order problem is derived by considering \eqref{viga2} with
		$$f(x,t) = t^{1/p}, \qquad Q(x,t)=\mu |t|^{s-1}t+ |t|^{q-1}t,$$
		where $\mu>0$, $p,\, q$ satisfies \eqref{pqHC} and $\frac{1}{p}\leq s<q$. More recently, a symmetric Hamiltonian system arising from a bilateral perturbation of a Lane-Emden system was considered in \cite{guimaraes2023hamiltonian}. More precisely, the authors proved the existence of classical positive solutions by using the Mountain Pass Theorem for \eqref{viga2} with
		$$f(x,t) = \mathcal{F}_{\lambda}^{-1}(t), \quad \mathcal{F}_{\lambda}(t)=t^{p}+\lambda t^r, \qquad Q(x,t)=\mu |t|^{s-1}t+ |t|^{q-1}t,$$
		where $\mu,\lambda>0,\, 1\leq rs, \, r<p,\, s<q$ and $p,q$ satisfy \eqref{pqHC}. Further advances in this direction were achieved in \cite{AgudeloRufVelez}, where the same equation from \cite{guimaraes2023hamiltonian} was considered but with $\mu,\lambda >0$, $0<r,s<1$ and $p,q>1$.

	\subsection{Main goal}
	Motivated by the previous discussion, the main goal of this work is to address some unexplored scenarios related to \eqref{viga2} that were not considered in \cite{ederson, EdmundsFortunatoJannelli, guimaraes2023hamiltonian, EdJe-djairo}. To do this, we will analyze the following Leray-Lions type problem:
	\begin{equation}\tag{P}\label{prob}  
		\Delta(f(x,\Delta u)) = \mu g(x)|u|^{s-1}u + |u|^{q-1}u \,\,\text{in } \Omega,
	\end{equation}
	with Navier boundary conditions
	\begin{equation}\tag{NBC}\label{navi}
		u = \Delta u = 0 \,\, \text{on } \partial\Omega,    
	\end{equation}
	or Dirichlet boundary conditions  
	\begin{equation}\tag{DBC}\label{Diri}
		u = \frac{\partial u}{\partial \mathbf{\eta}} = 0 \quad \text{on } \partial\Omega,
	\end{equation}
	where $\mathbf{\eta}$ denotes the outward unit normal vector to the boundary. Here, \( \mu > 0 \), \( g \in C^1(\Omega) \) is a positive function in \(\Omega\) and the pair \((p,q)\) belongs to the critical hyperbola given by \eqref{pqHC}.
	
	Considering a parameter \(r\) with \(0<r<p\), we also distinguish two different regimes for the perturbation exponent \(s\):  
	\begin{itemize}
		\item \textbf{Sublinear perturbation:}
		\begin{equation}\label{rssublinear}
			0<s< \frac{1}{p};
		\end{equation}
		\item \textbf{Superlinear perturbation:}
		\begin{equation}\label{rssuperlinear}
			\frac{1}{r}\leq s <q.
		\end{equation}
	\end{itemize}
	For each $t \in \mathbb{R}$, we assume that \(f(\cdot,t)\in C(\Omega)\) and, for each \(x\in \Omega\), the map \(f(x,\cdot) \in C(\mathbb{R})\) is odd and non-decreasing. Moreover, \(f\) satisfies the following conditions:  
	\begin{enumerate}
		\item[($f_1$)] \emph{Upper growth bound:} \( f(x,t)\leq t^{1/p} \) for every \(t>0\);  
		\item[($f_2$)] \emph{Asymptotic behavior at infinity:}  
		\[
		\lim_{t\to \infty}\frac{f(x,|t|)}{|t|^{1/p}}=1;
		\]
		\item[($f_3$)] \emph{Two-sided control near zero and at infinity:} there exist \(t_0 \geq 0\) and positive constants \(c_p,c_r\) such that 
		\[
		f(x,t)\geq \begin{cases}
			\displaystyle c_r t^{1/r},& 0<t\leq t_0,\\[0.3em]
			\displaystyle c_p t^{1/p},& t> t_0;
		\end{cases}
		\]
		\item[($f_4$)] \emph{Lower bound for primitive:} there exists \(c_q> \tfrac{1}{q+1}\) such that 
		$$F(x,t):=\int_0^t f(x,s)\,ds \geq\begin{cases}
			c_q f(x,t)t  & \text{if}\,\, \eqref{rssublinear}\,\, \text{holds}, \\[10pt]
			\frac{1}{s+1}f(x,t)t & \text{if}\,\, \eqref{rssuperlinear}\,\, \text{holds},
		\end{cases}$$
		for every $(x,t) \in \Omega\times \mathbb{R}$.
	\end{enumerate}

	\medskip 
	Typical examples for $f$ include:
	\begin{itemize}
		\item \(f(x,t) = |t|^{1/p-1}t\);
		\item $f(x,t)=\mathscr{F}^{-1}_\lambda(t)$ where $\mathscr{F}_\lambda(t) = \lambda |t|^{r-1}t + |t|^{p-1}t$;
		\item \(\displaystyle f(x,t) = \log(1+|t|)\, |t|^{\frac{1}{p}-2}t+\frac{|t|^\frac{1}{p}t}{1+|t|}\);
		\item $f(x,t)=\mathscr{G}^{-1}_\lambda(x,t)$ where $\mathscr{G}_\lambda(x,t) = g(x) + |t|^{p-1}t$;
		\item \(f(x,t) = \displaystyle|t|^{\frac{1}{p}-1}t\left(1-\frac{1}{2+|t|+|x|}\right)\).
	\end{itemize}

The above examples emphasize that the imposed assumptions constitute more than just technical constraints - they in fact encompass a large class of nonlinear operators of significant theoretical and applied interest. Moreover, assumption $(f_4)$ can be viewed as an inverse of the Ambrosetti-Rabinowitz condition \cite{ambrosetti1994combined} and it will be useful to prove the compactness of the energy functional associated to \eqref{prob}.

\subsection{Contributions and organization of the work}
	The existence, multiplicity, and regularity of solutions for the class of problems defined by \eqref{prob}-\eqref{navi} and \eqref{prob}-\eqref{Diri} are analyzed in this work. The results are structured as follows:
\begin{itemize}
\item In section \ref{pre}, the functional setting used throughout this work is introduced. The relevant Sobolev-type function spaces are defined, and the notation required for the subsequent variational approach are detailed.

\item In section \ref{ps-sec}, the compactness properties of the energy functional associated to \eqref{prob}-\eqref{navi} and \eqref{prob}-\eqref{Diri} are studied. Specific energy levels for which the energy functional satisfies the Palais-Smale condition are identified.

\item In section \ref{sec-sub}, the problems \eqref{prob}-\eqref{navi} and \eqref{prob}-\eqref{Diri} with sublinear nonlinearities are addressed. The existence of an unbounded sequence of distinct weak solutions is proven through the use of topological tools, specifically Krasnosel’skii’s genus and Clark’s deformation lemma. This reveals the inherent multiplicity of critical points.

\item In section \ref{sec-sup}, the superlinear case is examined. The existence of a nontrivial solution for \eqref{prob}-\eqref{navi} and \eqref{prob}-\eqref{Diri} is achieved by the Mountain Pass Theorem.

\item Finally, in section \ref{sec-connecHam}, the applicability of the main results is illustrated through the study of a class of second-order Hamiltonian systems.
\end{itemize}

	\section{Functional setting}\label{pre}
	
	For all $1\leq \theta \leq \infty$, we denote by $|\cdot|_{\theta}$ the usual norm of the space $L^{\theta}=L^{\theta}(\Omega)$ . We also consider the Banach space
	$$
	E_p := \begin{cases}
		W^{2,\frac{p+1}{p}}(\Omega) \cap W^{1,\frac{p+1}{p}}_0(\Omega) & \text{if}\,\, \eqref{navi} \,\,\text{holds},\\
		W^{2,\frac{p+1}{p}}_0(\Omega) & \text{if}\,\, \eqref{Diri}\,\, \text{holds},
	\end{cases}
	$$
	endowed with the norm
	$$
	\|u\| := |\Delta u|_{\frac{p+1}{p}}.
	$$
	In order to study the existence of weak solutions\footnote{A function \( u \in E_p \) is said to be a weak solution of \eqref{prob}-\eqref{navi} or \eqref{prob}-\eqref{Diri} if, for every \( \phi \in E_p \),
		\[
		\int_\Omega f(x,\Delta u) \Delta \phi \, dx = \int_\Omega \left( \mu g|u|^{s-1} u \phi + |u|^{q-1} u \phi \right) \, dx,
		\] 
		or equivalently, $u$ is a critical point of the functional $I_F$.} of \eqref{prob}-\eqref{navi} and \eqref{prob}-\eqref{Diri}, we define the $C^1$-functional:
	\begin{equation}\label{functional}
		I_F(u):=\displaystyle \int_{\Omega} F(x,\Delta u)dx-\frac{\mu}{s+1}  \int_{\Omega}g|u|^{s+1}dx -\frac{1}{q+1} |u|_{q+1}^{q+1}, \quad u \in E_p.
	\end{equation}
	
	\begin{lemma}\label{prop-fun-h}
		Under assumptions $(f_1)$-$(f_4)$, the functional $I_F(u)$ given by \eqref{functional} satisfies
		$$I_F(u)\geq h(\Vert u \Vert),$$
		with
		\begin{equation}\label{hh}
			h(t) = \begin{cases}
				\displaystyle (2|\Omega|)^{\frac{r-p}{r(p+1)}}C_r t^{\frac{r+1}{r}}-\frac{C_{g,s}\mu}{s+1} t^{s+1} -\frac{S^{-1}}{q+1}t^{q+1} & \text{if}\,\, t \leq  \left(\frac{C_p}{C_r}\right)^\frac{pr}{p-r}|\Omega|^{\frac{p}{p+1}},\\
				C_p  t^{\frac{p+1}{p}} \displaystyle -\frac{C_{g,s}\mu}{s+1} t^{s+1} -\frac{S^{-1}}{q+1}t^{q+1}& \text{if}\,\, t >  \left(\frac{C_p}{C_r}\right)^\frac{pr}{p-r}|\Omega|^{\frac{p}{p+1}},   
			\end{cases}
		\end{equation}
		where $C_p:=c_q\cdot c_p,~C_r:=\frac{r}{r+1}\cdot c_r$, $C_{g,s}=\Lambda_s^{-1}|g|_\infty$,
		\begin{equation*}\label{bestconstants}
			\Lambda_{\theta} := \inf_{\stackrel{u \in E_p,}{|u|_{\theta+1} = 1}} \|u\|,\quad 0<\theta<q,
		\end{equation*}   
		and $S:=\Lambda_q$.
	\end{lemma}
	
	\begin{proof}
		When $\|u\|=0$, the result is trivial. Given $u \in E_p\backslash\{0\}$, we set
		\begin{equation}\label{omegau}
			\omega_u:=\{x \in \Omega,\,\,|\Delta u(x)|\leq t_0\}.   
		\end{equation}
		Using $(f_3), ~(f_4)$, applying H\"{o}lder's inequality with 
		$$\frac{1}{\alpha}+\frac{\alpha-1}{\alpha}=1, \quad \alpha:=\frac{p(r+1)}{r(p+1)}>1,$$
		and exploiting the embedding $E_p\hookrightarrow L^{s+1}$ and $E_p\hookrightarrow L^{q+1}$, we obtain
		\begin{eqnarray*}
			I_F(u)  &= &  \int_{\omega_u}    F(x,\Delta u)dx+\displaystyle \int_{\Omega\backslash \omega_u} F(x,\Delta u)dx-\frac{\mu}{s+1} \int_{\Omega}g|u|^{s+1}dx -\frac{1}{q+1} |u|_{q+1}^{q+1}\vspace{5pt} \\
			&\geq &  C_r \int_{\omega_u}   |\Delta u|^\frac{r+1}{r} dx+ C_p \displaystyle \int_{\Omega\backslash \omega_u}   |\Delta u|^\frac{p+1}{p}dx 
			-\frac{C_{g,s}\mu}{s+1} \|u\|^{s+1} -\frac{S^{-1}}{q+1}\|u\|^{q+1}\\
			& \geq & \displaystyle |\omega_u|^{1-\alpha}C_r \left( \int_{\omega_u} |\Delta u|^\frac{p+1}{p} dx\right)^\alpha +  \displaystyle  C_p \int_{\Omega\backslash \omega_u} |\Delta u|^\frac{p+1}{p}dx \displaystyle -\frac{C_{g,s}\mu}{s+1} \|u\|^{s+1} \\
			&&-\frac{S^{-1}}{q+1}\|u\|^{q+1} \\
			&\geq & \displaystyle |\Omega|^{1-\alpha}C_r \left( \int_{\omega_u} |\Delta u|^\frac{p+1}{p} dx\right)^\alpha +  \displaystyle  C_p \int_{\Omega\backslash \omega_u} |\Delta u|^\frac{p+1}{p}dx \displaystyle -\frac{C_{g,s}\mu}{s+1} \|u\|^{s+1}\\
			&&-\frac{S^{-1}}{q+1}\|u\|^{q+1}.
		\end{eqnarray*}
		
		Now, we divide the analysis into two cases according to \eqref{hh}. 
		
\vspace{.8mm}		
		
\noindent $\bullet$ {\it Case 1.} If
		$$\Vert u\Vert  \leq \left(\frac{C_p}{C_r}\right)^\frac{pr}{p-r}|\Omega|^{\frac{p}{p+1}},$$
		we have
		\begin{eqnarray*}
			I_F(u)&\geq & |\Omega|^{1-\alpha}C_r\left[ \left( \int_{\omega_u} |\Delta u|^\frac{p+1}{p} dx\right)^\alpha + \left( \int_{\Omega\backslash \omega_u} |\Delta u|^\frac{p+1}{p} dx\right)^\alpha\right]\\
			&&-\frac{C_{g,s}\mu}{s+1} \|u\|^{s+1} -\frac{S^{-1}}{q+1}\|u\|^{q+1}\\
			&\geq &2^{1-\alpha}|\Omega|^{1-\alpha}C_r\|u\|^{\frac{r+1}{r}}-\frac{C_{g,s}\mu}{s+1} \|u\|^{s+1} -\frac{S^{-1}}{q+1}\|u\|^{q+1}\\
			& = & h(\|u\|).
		\end{eqnarray*}
		
\vspace{.8mm}		
		
\noindent $\bullet$ {\it Case 2.} Assuming that
		$$
		\displaystyle \Vert u\Vert  > \left(\frac{C_p}{C_r}\right)^\frac{pr}{p-r}|\Omega|^{\frac{p}{p+1}},
		$$
		then
		\begin{equation*}
			I_F(u)\geq C_p  \Vert u \Vert^{\frac{p+1}{p}} \displaystyle -\frac{C_{g,s}\mu}{s+1} \|u\|^{s+1} -\frac{S^{-1}}{q+1}\|u\|^{q+1}=h(\|u\|).
		\end{equation*}
		The proof is complete.
	\end{proof}


	\section{The $(PS)_c$ condition}\label{ps-sec}
	The main challenge in dealing with critical problems like \eqref{prob} lies in the lack of compactness arguments. Here, to overcome this difficulty, we will localize the energy levels where the functional $I_F$ satisfies the so-called Palais-Smale condition.
	
	\begin{definition}\label{PSc}
		A sequence $(u_n) \subset E_p$ is a Palais-Smale sequence for $I_F$ at level $c$ if
		\begin{itemize}
			\item $I_F(u_n) \to c$,
			\item $I_F'(u_n) \to 0 \in E_p^*$.
		\end{itemize}
		If every Palais-Smale sequence for $I_F$ at level $c$ has a convergent subsequence in $E_p$, we say that $I_F$ satisfies the Palais-Smale condition at level $c$ or, for short, the $(PS)_c$ condition. 
	\end{definition}
	
	The main result of this section reads as follows.
	
	\begin{prop}\label{propcompacidade}
		Under assumptions ($f_1$)-($f_4$), the functional $I_F$ given by \eqref{functional} satisfies the $(PS)_c$ condition with
		\begin{equation*}
			c <\begin{cases}
				\mathscr{C}_q S^{\frac{pN}{2(p+1)}} - k\mu^{\frac{q+1}{q - s}}, &\text{if } (r,s) \text{ satisfies }\eqref{rssublinear},\\
				\frac{2}{N} S^{\frac{pN}{2(p+1)}} &\text{if } (r,s) \text{ satisfies }\eqref{rssuperlinear},
			\end{cases} 
		\end{equation*}
		where $\mathscr{C}_{\theta} = c_q - \frac{1}{\theta+1}$, $\theta\in \mathbb{R}$, and
		\begin{equation}\label{kk}
			k:=|\Omega|(q-s)\left(\frac{|g|_\infty \mathscr{C}_s}{s+1} \right)^{\frac{q+1}{q-s}}\left(\frac{q+1}{\mathscr{C}_q}\right)^{\frac{s+1}{q-s}}>0. 
		\end{equation}
	\end{prop}
	
	The proof will be divided into several technical lemmas. Throughout the proofs in this section, we consider an arbitrary Palais-Smale sequence $(u_n)$ for $I_F$ at an arbitrary level $c$. We also maintain the notation $(u_n)$ for every subsequence of the original sequence.
	
	\begin{lemma}\label{psltd}
		$(u_n)$ is bounded in $E_p$.
	\end{lemma}
	\begin{proof}
		Let 
		$$
		\theta_*:=\begin{cases}
			q, &\text{if } (r,s) \text{ satisfies }\eqref{rssublinear},\\
			s, &\text{if } (r,s) \text{ satisfies }\eqref{rssuperlinear},
		\end{cases} 
		\quad \mbox{and} \quad   \mu_*:=\begin{cases}
			\frac{\mu(q-s)}{(s+1)(q+1)}, &\text{if } (r,s) \text{ satisfies }\eqref{rssublinear},\\
			0, &\text{if } (r,s) \text{ satisfies }\eqref{rssuperlinear}.
		\end{cases}  
		$$
		By definition, there exists a positive sequence $\varepsilon_n \rightarrow 0$ such that 
		\begin{eqnarray*}
			\varepsilon_n \Vert u_n\Vert + c &\geq & I_F(u_n)- \frac{1}{\theta_*+1} \langle I_F'(u_n),u_n\rangle\\
			&   \geq &\displaystyle \int_{\Omega}  \left(F(x,\Delta u_n)-\frac{1}{\theta_*+1} f(x,\Delta u_n) \Delta u_n \right) dx -\mu_*\int_{\Omega}g|u_n|^{s+1}dx.
		\end{eqnarray*} 
		
		Consider the function 
		$$H_p(x,t):=|t|^{-\frac{p+1}{p}}\left(F(x,t)-\frac{1}{\theta_*+1}f(x,t)t\right),\,\, x \in \Omega,\,\, t\in \mathbb{R}.$$
		From $(f_2)$ and by L'Hospital rule,
		\begin{equation*}
			\lim_{t\rightarrow \infty} H_p (x,t)=\lim_{t\rightarrow \infty}\left[ \frac{p}{p+1}\frac{f(x,t)}{t^{1/p}}-\frac{f(x,t)}{(\theta_*+1)t^{1/p}}\right]=\frac{p\theta_*-1}{(p+1)(\theta_*+1)}.
		\end{equation*}
		Consequently, for every $\epsilon>0$, there exists $t_0=t_0(\epsilon)>0$ such that
		$$H_p (x,t)> \underbrace{\frac{p\theta_*-1}{(p+1)(\theta_*+1)}-\epsilon}_{:=b_{\epsilon}} >0,\,\,\, x \in \Omega,\,\,\, |t|>t_0,$$
		which implies that
		\begin{equation}\label{Hp-ineq}
			F(x,t)-\frac{1}{\theta_*+1} f(x,t)t > b_{\epsilon}|t|^{\frac{p+1}{p}}, \,\,\, x \in \Omega,\,\,\, |t|>t_0.
		\end{equation}
		
		Therefore,
		\begin{eqnarray*}
			\varepsilon_n \Vert u_n\Vert + c &\geq & b_\epsilon \displaystyle \int_{\Omega\backslash \omega_u} |\Delta u_n|	^{\frac{p+1}{p}}dx -\mu_*C_{g,s}\Vert u_n\Vert^{s+1}\\
			&=& b_\epsilon \left(  \displaystyle \int_{\Omega} |\Delta u_n|^{\frac{p+1}{p}}dx - \displaystyle \int_{\omega_u} |\Delta u_n|^{\frac{p+1}{p}}dx\right)-\mu_*C_{g,s}\Vert u_n\Vert^{s+1}\\
			&\geq & b_\epsilon\Vert u_n\Vert ^{\frac{p+1}{p}}-\mu_*C_{g,s}\Vert u_n\Vert^{s+1}-b_\epsilon t_0^{\frac{p+1}{p}}|\Omega|.  
		\end{eqnarray*}
		Since $s<\frac{1}{p}$ when \eqref{rssublinear} holds, we conclude that $(u_n)$ is bounded in $E_p$.
	\end{proof}
	
	\subsection{A localizing result}
	In order to localize the levels where $I_F$ satisfies the $(PS)_c$ condition, we present a slight refinement of Lions Lemma \cite[Lemma I.1]{lions} to our case.
	
	\begin{lemma}\label{convseqlim}
		Let $(u_n)$ be a Palais-Smale sequence for $I_F$ at level $c$, up to a subsequence, we have:
		\begin{itemize}
			\item[\rm (i)] $u_n \rightharpoonup u$ in $E_p$.
			\item[\rm (ii)] $u_n \rightarrow u$ a.e. in $\Omega$ and in $L^{\theta}$, for all $1\leq \theta < q+1$.
			\item[\rm (iii)] $\left| \Delta u_n\right|^{\frac{p+1}{p}} \stackrel{*}{\rightharpoonup}\vartheta $ in the sense of measures on $\overline{\Omega}$.
			\item[\rm (iv)] $\left| u_n \right|^{q+1} \stackrel{*}{\rightharpoonup} \nu$ in the sense of measures on $\overline{\Omega}$.
			\item[\rm (v)] $\nabla u_n \rightharpoonup \nabla u$ in $\left( W^{1, \frac{p+1}{p}}(\Omega) \right)^N$.
			\item[\rm (vi)] $\nabla u_n \rightarrow \nabla u$ a.e. in $\Omega$ and in $\left( L^{\sigma} \right)^N$, for all $1 \leq \sigma < \sigma^*$, with $\sigma^* > \frac{p+1}{p}$ depending on the critical Sobolev embedding of $W^{1, \frac{p+1}{p}}(\Omega)$.
			\item[\rm (vii)] There exist an at most finite index set $J$, a family of points $\{ x_j: j \in J \} \subset \overline{\Omega}$ and a positive sequence $\{ \nu_j: j \in J \} ,\{ \vartheta_j: j \in J\}$ such that:
			\begin{itemize}
				\item[$a)$] \(
				\displaystyle \nu = \left| u \right|^{q+1} + \sum_{j \in J} \nu_j \delta_{x_j},\)
				\item[$b)$] \(\vartheta \geq \left| \Delta u \right|^{\frac{p+1}{p}} + \displaystyle\sum_{j\in J} \vartheta_j \delta_{x_j},
				\)
				\item[$c)$] $\nu_j \geq S^{\frac{pN}{2(p+1)}}$, for every $j \in J$.
			\end{itemize}
		\end{itemize}
	\end{lemma}
	\begin{proof}
		By Lemma \ref{psltd} and \cite[Lemma 3.3]{ederson}, conditions (i)-(vi) hold and there exist an at most countable set $J$,  a family of points $\{ x_j: j \in J \} \subset \overline{\Omega}$ and two positive sequences $\{ \nu_j: j \in J \} ,\{ \vartheta_j: j \in J\}$ such that
		\begin{itemize}
			\item[$a')$] \(
			\displaystyle \nu = \left| u \right|^{q+1} + \sum_{j \in J} \nu_j \delta_{x_j},\)
			\item[$b')$] \(\vartheta \geq \left| \Delta u \right|^{\frac{p+1}{p}} + \displaystyle\sum_{j\in J} \vartheta_j \delta_{x_j},
			\)
			\item[$c')$] \(
			\displaystyle {S \, \nu_j^{\frac{p+1}{p}\frac{1}{q+1}} \leq \vartheta_j \,\, \hbox{for all}\,\, j \in J. \,\, \hbox{In particular}\,\,\sum_{j \in J} \nu_j^{\frac{p+1}{p}\frac{1}{q+1}} < + \infty}.
			\)
		\end{itemize}
		Then, to show (vii), it is sufficient to verify that 
		\begin{equation}\label{nj}
			\nu_j \geq \vartheta_j, \quad \forall\,\, j \in J.   
		\end{equation} 
		Indeed, if \eqref{nj} holds, using $c')$ and noting that 
		$$\frac{pN}{2(p+1)}=\left(1-\frac{1}{q+1}\frac{p+1}{p}\right)^{-1},$$
		we deduce
		$$
		\nu_j=0 \text{ or }\nu_j \geq S^{\frac{pN}{2(p+1)}},
		$$
		which shows $c)$. Consequently, the convergence 
		$$\sum_{j\in J} \nu_j^{\frac{p+1}{p}\frac{1}{q+1}}<+\infty,$$
		implies that $J$ is at most finite. Then, $a)$ and $b)$ are proved.\\

		\noindent\textit{Proof of \eqref{nj}:} Let $\zeta \in C^{\infty}_{c}(\mathbb{R}^{N} )$ be a cut-off function such that
		$$0\leq \zeta \leq 1, \quad \zeta\equiv1 \,\,\mbox{in}\,\, B(0,1), \quad \operatorname{supp}(\zeta) \subset B(0,2).$$ 
		For each $\beta >0$, we set 
		\begin{equation}\label{beta}
			\zeta_\beta(x) := \zeta\left(\frac{x-x_j}{\beta}\right). 
		\end{equation}
		By definition, there exists $C>0$, independent of  $\beta$, such that
		$$ |\nabla \zeta_\beta (x)| \leq \frac{C}{\beta}, \  |\Delta \zeta_\beta (x)| \leq \frac{C}{\beta^2}, \ \forall\, x \in \mathbb{R}^N.$$
		By  \cite[Lemma 3.4]{ederson}, $\zeta_\beta u_n  \in E_p$, for all $n \in \mathbb{N}$ and $\beta>0$. Then,
		\begin{equation}\label{notas 1}
			\,\langle I_F'(u_n),\zeta_\beta u_n\rangle =\sum_{l=1}^{5}a_{l,n,\beta},
		\end{equation}
		where
		\begin{eqnarray}{l}
			\displaystyle a_{1,n,\beta}:= \displaystyle \int_{\overline{\Omega}} f(x,\Delta u_n) u_n \Delta \zeta_\beta dx,\quad
			a_{2,n,\beta}:=- \mu \int_{\overline{\Omega}} g|u_n|^{s+1}\zeta_\beta dx, \quad
			a_{3,n,\beta}:= -  \int_{\overline{\Omega}} |u_n|^{q+1} \zeta_\beta dx,\\[6pt] 
			\displaystyle a_{4,n,\beta}:= 2 \displaystyle \int_{\overline{\Omega}} f(x,\Delta u_n) \nabla u_n \nabla \zeta_\beta dx,\quad
			a_{5,n,\beta}:=\int_{\overline{\Omega}} f(x,\Delta u_n)\Delta u_n \zeta_\beta dx.  
		\end{eqnarray}
		Next, we analyze separately the limits of each double sequence $a_{l,n,\beta}$, $l=1,\hdots,5$.
		
		\vspace{1.5mm}
		
		\noindent {\bf{Limit of $a_{1,n,\beta}$}:} Recalling assumption $(f_1)$ and applying Hölder inequality with
		$$\frac{1}{p+1}+ \frac{1}{q+1} + \frac{2}{N}=1,$$
		we have
		$$
		\displaystyle \left| a_{1,n,\beta} \right| \leq  \int_{\overline{\Omega}} |\Delta u_n|^{1/p} |u_n| |\Delta \zeta_\beta| dx \leq C \!\left( \!\displaystyle \int_{\overline{\Omega}} |u_n|^{q+1} \left|\Delta \zeta\left(\frac{x-x_j}{\beta}\right)\right|^{\frac{q+1}{2}} \!dx\right)^{\frac{1}{q+1}},$$
		for some $C>0$, independent of $n$ and $\beta$.
		Using (iv) and applying the Dominated Convergence Theorem, we infer
		$$
		\displaystyle \int_{\overline{\Omega}} |u_n|^{q+1}  \left|\Delta \zeta\left( \frac{x-x_j}{\beta}\right)\right|^{\frac{q+1}{2}}  dx \xrightarrow{n \rightarrow \infty} \displaystyle \int_{\overline{\Omega}}   \left|\Delta\zeta\left( \frac{x-x_j}{\beta} \right)\right|^{\frac{q+1}{2}}  d\nu \xrightarrow{\beta \rightarrow 0} 0.
		$$
		Hence,
		\begin{equation}\label{limit1}
			\lim_{\beta\to 0}\lim_{n\to \infty}a_{1,n,\beta}=0. 
		\end{equation}
		
		\vspace{1.5mm}
		
		\noindent {\bf{Limit of $a_{2,n,\beta}$}:} Taking into account that $E_p\stackrel{c}{\hookrightarrow} L^{s+1}$,
		$$u_n \rightharpoonup u\,\, {\text{in}}\,\, E_p \qquad\text{and}\qquad \zeta_{\beta} (x) \xrightarrow{\beta \rightarrow 0} \delta_{x_j} (x), \quad \forall \,\,x \in \Omega,$$
		we get
		$$ 
		\displaystyle \int_{\overline{\Omega}} g|u_n|^{s+1}\zeta_\beta dx \xrightarrow{n \rightarrow \infty}  \displaystyle \int_{\overline{\Omega}} g|u|^{s+1} \zeta_\beta dx \xrightarrow{\beta \rightarrow 0} 0.
		$$
		Then,
		\begin{equation}\label{limit2}
			\lim_{\beta\to 0}\lim_{n\to\infty}a_{2,n,\beta}=0. 
		\end{equation}
		
		\vspace{1.5mm}
		
		\noindent {\bf{Limit of $a_{3,n,\beta}$}:} Using $a')$, we get
		\begin{equation}\label{limit3}
			\lim_{\beta\to 0}\lim_{n\to\infty}a_{3,n,\beta}= -\nu_j.
		\end{equation}
		
		\vspace{1.5mm}
		
		\noindent {\bf{Limit of $a_{4,n,\beta}$}:} Note that
		$$
		\left| \displaystyle \int_{\overline{\Omega}} f(x,\Delta u_n) \nabla u_n \nabla \zeta_\beta dx \right| \leq C \left(\displaystyle \int_{\overline{\Omega}}  \left(\frac{1}{\beta} \left| \nabla\zeta\left(\frac{x-x_j}{\beta}\right)\right| |\nabla u_n|\right)^{\frac{p+1}{p}}dx\right)^{\frac{p}{p+1}}.$$
		Since $p$ lies on the critical hyperbola, it follows that $\frac{p}{p+1} < \frac{N}{2}$. Consequently,
		$$
		\displaystyle \int_{\overline{\Omega}} \!\left(\frac{1}{\beta} \left| \nabla\zeta\left(\frac{x-x_j}{\beta}\right)\right| |\nabla u_n|\right)^{\frac{p+1}{p}}\!\!dx \xrightarrow{n \rightarrow \infty} \displaystyle \int_{\overline{\Omega}} \!\left(\frac{1}{\beta} \left| \nabla\zeta\left(\frac{x-x_j}{\beta}\right)\right| |\nabla u|\right)^{\frac{p+1}{p}}\!\!dx\xrightarrow{\beta \rightarrow 0} 0.$$
		Hence,
		\begin{equation}\label{limit4}
			\lim_{\beta\to 0}\lim_{n\to\infty}a_{4,n,\beta}= 0.
		\end{equation}
		
		\vspace{1.5mm}
		
		\noindent {\bf{Limit of $a_{5,n,\beta}$}:} Let $\varepsilon>0$. By condition ($f_{2}$), there exists $M=M(\varepsilon)>0$ such that 
		$$f(x,t)t\geq \frac{1}{1+\varepsilon}|t|^\frac{p+1}{p}, \quad\forall\,\,|t|\geq M.$$ 
		For $n>0$, we define
		\begin{equation}\label{An}
			A_n(M):=\{x\in B(x_j,2\beta)\cap \overline{\Omega}; |\Delta u_n(x)|\geq M\}, \quad B_n(M):= (\overline{\Omega}\cap B(x_j,2\beta))\backslash A_n(M). 
		\end{equation}
		Then,
		\begin{eqnarray*}
			\int_{\overline{\Omega}} f(x,\Delta u_n)\Delta u_n \zeta_\beta dx
			&\!=\!&\displaystyle \int_{A_n(M)} f(x,\Delta u_n)\Delta u_n \zeta_\beta dx+\displaystyle \int_{B_n(M)}f(x,\Delta u_n)\Delta u_n \zeta_\beta dx \\
			&\!\geq\! &\!\frac{1}{1+\varepsilon} \displaystyle \int_{\overline{\Omega}} |\Delta u_n|^{\frac{p+1}{p}} \zeta_\beta dx \\
			&&+ \displaystyle \int_{B_n(M)}\!\! \left(f(x,\Delta u_n)\Delta u_n- \frac{1}{1+\varepsilon} |\Delta u_n|^{\frac{p+1}{p}}\right)\!  \zeta_\beta dx.
		\end{eqnarray*}
		Since
		\begin{multline*}
			\lim_{\beta\rightarrow 0} \limsup_{n\rightarrow \infty} \left|\displaystyle \int_{B_n(M)} \left(f(x,\Delta u_n)\Delta u_n- \frac{1}{1+\varepsilon} |\Delta u_n|^{\frac{p+1}{p}}\right) \zeta_\beta dx \right|\\
			\leq \lim_{\beta\rightarrow 0} \limsup_{n\rightarrow \infty} \displaystyle \int_{B_n(M)} \left(f(x,M)M+\frac{1}{1+\varepsilon} M^{\frac{p+1}{p}}\right) \zeta_\beta dx  = 0,
		\end{multline*}
		we arrive at
		\begin{equation}\label{limit5}
			\lim_{\beta\to 0}\limsup_{n\rightarrow \infty}a_{5,n,\beta}\geq \frac{1}{1+\varepsilon}\vartheta_j. 
		\end{equation}
		
		In view of convergences \eqref{limit1}-\eqref{limit5}, we can pass the limit in \eqref{notas 1} to obtain
		$$
		\nu_j \geq \frac{1}{1+\varepsilon}\vartheta_j, \quad \forall \,\, \varepsilon>0,
		$$
		which implies that $\nu_j \geq \vartheta_j$. The proof is now complete. 
	\end{proof}
	
	We have the following consequence in the superlinear case.
	
	\begin{Corollary}\label{ineqnuk}
		Assume that \eqref{rssuperlinear} holds. Then, up to a subsequence, we have
		\begin{equation*}
			\lim_{n\rightarrow \infty} \int_{\overline{\Omega}} \left[F(x,\Delta u_n)-\frac{1}{s+1} f(x,\Delta u_n) \Delta u_n \right] dx \geq \frac{ps-1}{(p+1)(s+1)} \vartheta_j, \quad j \in J,
		\end{equation*}
		where $s$ is as in the proof of Lemma \ref{psltd}.
	\end{Corollary}
	\begin{proof}
		
		From \eqref{Hp-ineq} and $(f_4)$, we get
		\begin{eqnarray}\label{parte1eq}
			&&\nonumber \int_{\overline{\Omega}} \left[F(x,\Delta u_n)-\frac{1}{s+1} f(x,\Delta u_n) \Delta u_n \right] dx\\
			&\geq & \displaystyle \int_{B_n(t_0)}\left[ F(x,\Delta u_n)-\frac{1}{s+1} f(x,\Delta u_n) \Delta u_n - b_{\epsilon}|\Delta u_n|^{\frac{p+1}{p}} \right]\zeta_{\beta} dx\\
			\nonumber&&+\displaystyle b_{\epsilon} \left(\displaystyle \int_{A_n(t_0)} |\Delta u_n|^{\frac{p+1}{p}} \zeta_{\beta} dx+
			\displaystyle \int_{B_n(t_0)}|\Delta u_n|^{\frac{p+1}{p}} \zeta_{\beta} dx\right) \\
			\nonumber&=&  \displaystyle \int_{B_n(t_0)}\left[F(x,\Delta u_n)-\frac{1}{s+1} f(x,\Delta u_n) \Delta u_n - b_{\epsilon}|\Delta u_n|^{\frac{p+1}{p}} \right]\zeta_{\beta} dx +\displaystyle b_{\epsilon} \displaystyle \int_{\overline\Omega} |\Delta u_n|^{\frac{p+1}{p}} \zeta_{\beta} dx,
		\end{eqnarray}
		where $\zeta_\beta$, $A_n(t_0)$ and $B_n(t_0)$ are given by \eqref{beta} and \eqref{An}, respectively. Now, note that
		\begin{multline}\label{parte2eq}
			\lim_{\beta \to 0}\limsup_{n\rightarrow \infty}\left| \displaystyle \int_{B_n(t_0)}\left[ F(x,\Delta u_n)-\frac{1}{s+1} f(x,\Delta u_n) \Delta u_n  - b_{\epsilon}|\Delta u_n|^{\frac{p+1}{p}} \right]\zeta_{\beta} dx\right|\\ 
			\leq \lim_{\beta \to 0}\limsup_{n\rightarrow \infty} \displaystyle \int_{\overline\Omega}\left[ F(x,t_0)+\frac{1}{s+1} f(x,t_0) t_0 + b_{\epsilon}|t_0|^{\frac{p+1}{p}}  \right]\zeta_{\beta} dx =0.
		\end{multline}
		From  Lemma \ref{convseqlim}, we also have
		\begin{equation}\label{parte3eq}
			\displaystyle \lim_{\beta \to 0}\limsup_{n\rightarrow \infty}b_{\epsilon} \displaystyle \int_{\overline\Omega} |\Delta u_n|^{\frac{p+1}{p}} \zeta_{\beta} dx\geq b_\epsilon\vartheta_j.
		\end{equation}
		Hence, \eqref{parte1eq}, \eqref{parte2eq},  \eqref{parte3eq} and the arbitrariness of $\epsilon>0$ yield the desired inequality.
	\end{proof}
	
	\begin{Corollary}
		\label{empty}
		If $J\neq \emptyset$, then
		$$   c\geq\begin{cases}
			\mathscr{C}_q S^{\frac{pN}{2(p+1)}} - k\mu^{\frac{q+1}{q - s}}, &\text{if } (r,s) \text{ satisfies }\eqref{rssublinear},\\
			\frac{2}{N} S^{\frac{pN}{2(p+1)}} &\text{if } (r,s) \text{ satisfies }\eqref{rssuperlinear},
		\end{cases}$$
		where $k>0$ is given by \eqref{kk}.
	\end{Corollary}
	
	\begin{proof}
		Exploring the boundedness of $(u_n)$ in $E_p$, we get 
		\begin{equation}\label{eqseqlim}
			\lim_{n\to +\infty}\int_{\overline{\Omega}} f(x,\Delta u_n)\Delta u_n = \mu \int_{\overline{\Omega}}g|u|^{s+1}dx+|u|_{q+1}^{q+1}+\sum_{j \in J}\nu_j.
		\end{equation}
		
		First, we address the sublinear case. If \eqref{rssublinear} holds, we may invoke ($f_4$), \eqref{eqseqlim} and Lemma \ref{convseqlim} to obtain
		\begin{eqnarray*}
			\nonumber  c&=&\lim_{n\rightarrow \infty} I_F(u_n)\\
			\nonumber &=& \lim_{n\rightarrow \infty} \left[\displaystyle \int_{\overline{\Omega}} F(x,\Delta u_n)dx -\frac{\mu}{s+1}\int_{\overline{\Omega}}g|u_n|^{s+1}dx-\frac{1}{q+1}  |u_n|_{q+1}^{q+1}\right]\\
			\nonumber &\geq&\displaystyle\lim_{n\rightarrow \infty}\left[c_q \displaystyle \int_{\overline{\Omega}} f(x,\Delta u_n)\Delta u_n  -\frac{\mu}{s+1}\int_{\overline{\Omega}}g|u_n|^{s+1}dx-\frac{1}{q+1}  |u_n|_{q+1}^{q+1}\right]\\
			\nonumber &=&\displaystyle c_q\left(\mu\int_{\overline{\Omega}}g|u|^{s+1}dx+|u|_{q+1}
			^{q+1}+\sum_{j \in J}\nu_j\right)\\ & &-\frac{\mu}{s+1}\int_{\overline{\Omega}}g|u|^{s+1}dx-\frac{1}{q+1}|u|_{q+1}
			^{q+1}-\frac{1}{q+1}\sum_{j \in J}\nu_j\\
			&\geq&\displaystyle\left(c_q-\frac{1}{s+1}\right)\mu|g|_\infty|u|_{s+1}^{s+1}+\left(c_q-\frac{1}{q+1}\right)\left(|u|_{q+1}
			^{q+1}+\sum_{j \in J}\nu_j\right)\\
			&\geq & \Psi(|u|_{q+1})+\mathscr{C}_q S^{\frac{pN}{2(p+1)}}.
		\end{eqnarray*}
		where 
		$$ \Psi(t):= \mathscr{C}_q t^{q+1}-b_0\mu t^{s+1}, \quad b_0:=-|g|_\infty|\Omega|^{\frac{q-s}{q+1}}\mathscr{C}_s.$$
		Noting that $\Psi$ has a global minimum at 
		$$r_0=\begin{cases}
			0& \text{ if } b_0\leq0,\\
			\left(\frac{N(s+1)b_0\mu}{2 (q+1)}\right)^{\frac{1}{q-s}}&\text{ if } b_0>0,
		\end{cases}$$
		we get
		$$\Psi(t)\geq \Psi(r_0)\geq -k\mu^{\frac{q+1}{q-s}},$$ 
		where $k>0$ is given by \eqref{kk}.
		Hence,
		$$c\geq \mathscr{C}_q S^{\frac{pN}{2(p+1)}}-k\mu^{\frac{q+1}{q-s}}.$$
		
		Now, let us assume that \eqref{rssuperlinear} holds. From Lemma \ref{convseqlim} and Corollary \ref{ineqnuk}, we deduce
		\begin{eqnarray*}
			\nonumber  c&=&\lim_{n\rightarrow \infty} I_F(u_n)- \frac{I_F'(u_n)u_n}{s+1}\\ 
			&=&\lim_{n\rightarrow \infty}\displaystyle \int_{\Omega}  \left[F(x,\Delta u_n)-\frac{1}{s+1}f(x,\Delta u_n) \Delta u_n \right]  dx +\left(\frac{1}{s+1}-\frac{1}{q+1}\right)|u_n|_{q+1}^{q+1} \\
			&\geq &\left( \frac{p}{p+1}-\frac{1}{s+1}\right)\nu_j +\lim_{n\rightarrow \infty}\left(\frac{1}{s+1}-\frac{1}{q+1}\right)|u_n|_{q+1}^{q+1}\\
			&\geq& \left( \frac{p}{p+1}-\frac{1}{q+1}\right) \nu_j\\
			&\geq& \frac{2}{N} S^{\frac{pN}{2(p+1)}}.
		\end{eqnarray*}
		The proof is now complete.
	\end{proof}   
	
	\begin{obs}\label{point}
		{\rm Corollary \ref{empty}  reveals the location of levels $c$ where $\nu$ has no singular points. This will be crucial to conclude the convergence $u_n \to u$ in $E_p$, whose details are given below.}  
	\end{obs}
	\subsection{Convergence results}
	
	\begin{lemma}\label{ra}
		Under the notations of Lemma \ref{convseqlim}, for every $K \subset\subset \Omega \backslash\{x_j: j \in J \}$, we have
		\begin{equation*}
			\displaystyle \int_K (f(x,\Delta u_n)-f(x,\Delta u))(\Delta u_n - \Delta u)dx \to0,
		\end{equation*}
		up to a subsequence.
	\end{lemma}
	\begin{proof} Let $\delta = dist(K, \{x_j: j \in J \})$. For each $\theta \in (0, \delta)$, consider 
		$$D_{\theta} = \{ x \in \Omega: dist(x, K)< \theta\}$$
		and $\xi_\theta \in C_c^{\infty}(\Omega)$ satisfying
		$$0\leq \xi_{\theta} \leq 1, \qquad \xi_{\theta}\equiv 1 \,\, \text{on}\,\, D_{\theta/2}, \qquad \xi_{\theta} \equiv 0 \,\, \text{on} \,\, \Omega \backslash D_{\theta}.$$
		By the monotonicity of $f$, we have 
		\begin{eqnarray*}\label{integrand1}
			0 &\leq& \displaystyle \int_K (f(x,\Delta u_n)-f(x,\Delta u))(\Delta u_n - \Delta u)dx \\
			&\leq & \displaystyle \int_{\overline{\Omega}} (f(x,\Delta u_n)-f(x,\Delta u))(\Delta u_n - \Delta u)\xi_\theta dx \\ 
			&=&\displaystyle \int_{\overline{\Omega}} f(x,\Delta u_n)\Delta u_n \xi_\theta - f (x,\Delta u_n) \Delta u \xi_\theta - f(x,\Delta u)(\Delta u_n - \Delta u) \xi_\theta dx.
		\end{eqnarray*}
		
		On the other hand, the uniform boundedness of $(u_n \xi_\theta)$ and $(u \xi_\theta)$ in $E_p$ yields
		\begin{eqnarray*}
			0 & = & \lim_{n\to \infty}\langle I_F'(u_n), \xi_\theta (u_n-u)\rangle\\
			& = & \lim_{n\to \infty}\int_{\overline{\Omega}} f(x,\Delta u_n)(\Delta (u_n-u) \xi_\theta  + (u_n-u) \Delta \xi_\theta + 2 \nabla (u_n-u) \nabla \xi_\theta) dx\\
			&&-\lim_{n\to \infty}\left[\int_{\overline{\Omega}} |u_n|^{q-1}u_n \xi_\theta u dx - \!\displaystyle \int_{\overline{\Omega}} |u_n|^{q+1} \xi_\theta dx + \mu\! \!\displaystyle \int_{\overline{\Omega}} g|u_n|^{s-1} u_n \xi_\theta u dx \right]\\
			&& +\mu\lim_{n\to \infty} \!\displaystyle \int_{\overline{\Omega}} g|u_n|^{s+1} \xi_\theta dx.
		\end{eqnarray*}
		Now, applying H\" older's inequality and \cite[Lemma 2.4 and Lemma 3.6]{ederson}, we deduce 
		\begin{eqnarray*}
			0 &\leq& \int_K (f(x,\Delta u_n)-f(x,\Delta u))(\Delta u_n - \Delta u)dx \\
			&\leq &  \int_{\overline{\Omega}} f(x,\Delta u_n)\Delta\xi_\theta (u_n - u)dx\\
			& & +2\displaystyle \int_{\overline{\Omega}} f(x,\Delta u_n)  \nabla \xi_\theta \nabla (u_n -u)dx-\displaystyle \int_{\overline{\Omega}} f(x,\Delta u) (\Delta u_n - \Delta u)\xi_\theta dx  \\
			&\leq& \displaystyle C\left(\int_{D_\theta} |u_n-u|^{q+1}dx\right)^{\frac{1}{q+1}}+C\left(\int_{\overline{\Omega}} |\nabla u_n-\nabla u|^{\frac{p+1}{p}}dx\right)^{\frac{p}{p+1}}\\
			& & -\displaystyle \int_{\overline{\Omega}} f(x,\Delta u) (\Delta u_n - \Delta u)\xi_\theta dx,    
		\end{eqnarray*}
		for some $C>0$. Hence, the desired convergence follows from Lemma \ref{convseqlim}.
	\end{proof}

	\begin{lemma}
		Up to a subsequence,  $\Delta u_n \to \Delta u$ a.e. in $\Omega$.
	\end{lemma}
	\begin{proof}
		Let $K\subset\subset \Omega\backslash\{ x_j: j \in J \}$. From Lemma \ref{ra} and by the Inverse Dominated Convergence Theorem, there exists a subsequence $(u_n)$ such that
		$$(f(x,\Delta u_n)-f(x,\Delta u))(\Delta u_n - \Delta u)\to 0 \,\, \,\,\text{a.e. in} \,\, \, K.$$ 
		Using \cite[Lemma 6]{masomurat}, we get $\Delta u_n \rightarrow \Delta u$ a.e.  in  $K$. Since $K$ is an arbitrary compact subset of $\Omega\backslash\{ x_j: j \in J \}$, we conclude that $\Delta u_n \rightarrow \Delta u$ a.e.  in $\Omega$.
	\end{proof}
	
	\begin{lemma}
		Up to a subsequence, $f(x,\Delta u_n) \rightharpoonup f(x,\Delta u) \text{ in } L^{p+1}$.
	\end{lemma}
	\begin{proof}
		Up to a subsequence, we have
		\begin{equation*}
			\left\{
			\begin{array}{lll}
				\Delta u_n \to \Delta u \ \  \text{a.e. in }\Omega, \\ (\Delta u_n)\text{ is bounded in } L^{\frac{p+1}{p}},  \\
				|f(x,\Delta u_n)|\leq |\Delta u_n|^{1/p}.
			\end{array}
			\right.
		\end{equation*}
		Then, $f(x,\Delta u_n) \to f(x,\Delta u)$ a.e. in $\Omega$ and $(f(x,\Delta u_n))$ is bounded in $L^{p+1}$. Hence, $f(x,\Delta u_n) \rightharpoonup f(x,\Delta u) \text{ in } L^{p+1}$.
	\end{proof}
	
	\begin{lemma}
		\label{propfracasol}
		For every $w\in E_p$, $ \langle I_F'(u), w \rangle=0$.
	\end{lemma}
	
	\begin{proof}
		By definition of Palais-Smale sequence for $I_F$, we have $\langle I_F'(u_n), w \rangle \rightarrow 0$, for all $w \in E_p$. On the other hand, up to a subsequence, 
		\begin{equation*}
			\left\{
			\begin{array}{lll}
				f(x,\Delta u_n) \rightharpoonup f(x,\Delta u) \text{ in } L^{p+1},\\
				|u_n|^{q-1}u_n \rightharpoonup |u|^{q-1}u \text{ in } L^{\frac{q+1}{q}},\\ |u_n|^{s-1}u_n \rightarrow |u|^{s-1}u \text{ in } L^{\frac{s+1}{s}}.
			\end{array}
			\right.
		\end{equation*}
		Thus, for all $w \in E_p$, $\langle I_F'(u_n), w \rangle \rightarrow \langle I_F'(u), w \rangle,$ that is $\langle I_F'(u), w \rangle=0$.
	\end{proof}
	
	
	\subsection{Proof of Proposition \ref{propcompacidade}}
	Let $(u_n)$ be a Palais-Smale sequence for $I_F$ at level $c$ with  
	\begin{equation*}
		c <\begin{cases}
			\mathscr{C}_q S^{\frac{pN}{2(p+1)}} - k\mu^{\frac{q+1}{q - s}}, &\text{ if } (r,s) \text{ satisfies }\eqref{rssublinear},\\
			\frac{2}{N} S^{\frac{pN}{2(p+1)}}, & \text{ if } (r,s) \text{ satisfies }\eqref{rssuperlinear}.
		\end{cases}
	\end{equation*}
	where $k$ is given by \eqref{kk}. Then, Remark \ref{point} ensures that $\nu$ has no singular points, which allows for the possibility of obtaining the strong convergence \( u_n \to u \) in \( E_p \).  
	
	Using Lemma \ref{convseqlim}, \cite[Lemma 3.6]{ederson} and exploring the uniform convexity of $L^{q+1}$, we get $u_n \rightarrow u$ in $L^{q+1}$. Then, setting  $v_n:=u_n-u$, we obtain
	$$\begin{cases}
		v_n\rightharpoonup 0 \quad \text{ in} \quad E_p,\\
		\Delta v_n \rightarrow 0\quad {\rm a.e.\,\, in} \quad \Omega,\\
		v_n \rightarrow 0 \quad \text{in} \quad L^{q+1}.
	\end{cases}$$
	Now, noting that $\langle I_F'(u), u \rangle=0$ and
	$$
	|(a+b)f(x,a+b)-af(x,a)|\leq |a+b|^{\frac{p+1}{p}}+|a|^{\frac{p+1}{p}} \leq 2^p( |b|^{\frac{p+1}{p}} + |a|^{\frac{p+1}{p}}), \ \ \forall\,\, a,b \in \mathbb{R},\,\, x \in \Omega,
	$$
	we can apply \cite[Theorem 2]{brezislieb} with  $j_x(t)=tf(x,t)$, to get 
	\begin{eqnarray*}
		0 & = & \lim_{n\to \infty} \langle I_F'(u_n), u_n \rangle \\
		&= & \lim_{n\to \infty} \int_{\Omega} \left[f(x,\Delta u_n)\Delta u_n  -  |u_n|^{q+1} - \mu  g|u_n|^{s+1} \right]dx\\
		&=& \lim_{n\to \infty}\int_\Omega f(x,\Delta v_n)\Delta v_n dx.
	\end{eqnarray*}
	Therefore, by $(f_3)$ and using H\"{o}lder's inequality with
	$$\frac{1}{\alpha}+\frac{\alpha-1}{\alpha}=1, \quad \alpha:=\frac{p(r+1)}{r(p+1)}>1,$$
	we have
	\begin{eqnarray*}
		0 & = &\lim_{n\to \infty} \int_\Omega f(x,\Delta v_n)\Delta v_n dx \\
		&\geq & c_p \lim_{n\to \infty} \int_{\Omega\backslash\omega_{v_n}} |\Delta v_n|^{\frac{p+1}{p}} dx + c_r \lim_{n\to \infty} \int_{\omega_{v_n}} |\Delta v_n|^{\frac{r+1}{r}} dx \vspace{5pt} \\ 
		&\geq & c_p \lim_{n\to \infty} \int_{\Omega\backslash\omega_{v_n}} |\Delta v_n|^{\frac{p+1}{p}} dx + c_r\lim_{n\to \infty} |\Omega|^{1-\alpha} \left(\int_{\omega_{v_n}} |\Delta v_n|^{\frac{p+1}{p}} dx\right)^\alpha\\
		&\geq& 0,
	\end{eqnarray*}
	where $\omega_{v_n}$ is given by \eqref{omegau}. Hence, $u_n \to u$ in $E_p$. This concludes the proof of Proposition \ref{propcompacidade}. \qed

	\section{The sublinear case: multiplicity of weak solutions}\label{sec-sub}
	In order to obtain the existence of infinitely many weak solutions for \eqref{prob}-\eqref{navi} and \eqref{prob}-\eqref{Diri} with \eqref{rssublinear}, we must find a (possibly finite) sequence $(c_l)$ of critical values of $I_F$ corresponding to an infinite sequence $(u_k) \subset E_p$ of critical points. Precisely, our main result reads as follows.
	\begin{theorem}\label{theo1}
		Under assumptions \eqref{rssublinear} and ($f_1$)-($f_4$), there exists $\mu_0 > 0$ such that \eqref{prob}-\eqref{navi} and \eqref{prob}-\eqref{Diri} respectively admit infinitely many weak solutions for every $\mu \in (0, \mu_0)$. 
	\end{theorem}
	
	The proof of Theorem \ref{theo1} relies on the construction of a suitable truncated energy functional $\overline{I}_F$ such that its negative critical values correspond to the critical values of $I_F$. 
	
	\subsection{The truncated functional}
	
	Let $\mu>0$ be small enough such that the function $h$ defined by \eqref{hh} attains its positive maximum. 
	Let \( R_0 \) and \( R_1 \) be the first and second positive roots of \( h(t) \), respectively. Inspired by the ideas from \cite{AP1}, we consider a non-increasing \( C^{\infty} \) function \( \tau: \mathbb{R}^{+} \rightarrow [0, 1] \) satisfying
	\[
	\tau(t) = \left\{\begin{array}{l}
		1, \text { if } t \leqslant R_0 \\
		0, \text { if } t \geqslant R_1,
	\end{array}\right.
	\]
	and set 
	\begin{equation}\label{truncatedfunctional}
		\overline{I}_F(u)=\displaystyle \int_{\Omega}\left[ F(x,\Delta u)-\frac{\mu}{s+1} g |u|^{s+1} -\frac{1}{q+1}   |u|^{q+1}\varphi(u)\right]dx,
	\end{equation}
	where \( \varphi(u) := \tau(\|u\|) \).
	Repeating verbatim the arguments used in the proof of Proposition \ref{prop-fun-h}, we arrive at
	$$\overline{I}_F(u)\geq \overline{h}(\Vert u\Vert),$$
	with
	$$
	\overline{h}(t) = \begin{cases}
		\displaystyle(2|\Omega|)^{\frac{r-p}{r(p+1)}}C_r t^{\frac{r+1}{r}}-\frac{C_{g,s}\mu}{s+1} t^{s+1} -\frac{S^{-1}}{q+1}\tau(t)t^{q+1},& \text{ if } t \leq  \left(\frac{C_p}{C_r}\right)^\frac{pr}{p-r}|\Omega|^\frac{p}{p+1}\\
		C_p  t^{\frac{p+1}{p}} \displaystyle -\frac{C_{g,s}\mu}{s+1} t^{s+1} -\frac{S^{-1}}{q+1}\tau(t)t^{q+1},&  \text{ if } t >  \left(\frac{C_p}{C_r}\right)^\frac{pr}{p-r}|\Omega|^\frac{p}{p+1},
	\end{cases}.
	$$
	In particular, $\overline{h}=h$ when $t\leq R_0$, while 
	$$
	\overline{h}(t) = \begin{cases}
		\displaystyle	(2|\Omega|)^{\frac{r-p}{r(p+1)}}C_r t^{\frac{r+1}{r}}-\frac{C_{g,s}\mu}{s+1} t^{s+1},& \text{ if } t \leq  \left(\frac{C_p}{C_r}\right)^\frac{pr}{p-r}|\Omega|^\frac{p}{p+1}\\
		C_p  t^{\frac{p+1}{p}} \displaystyle -\frac{C_{g,s}\mu}{s+1} t^{s+1},&  \text{ if } t >  \left(\frac{C_p}{C_r}\right)^\frac{pr}{p-r}|\Omega|^\frac{p}{p+1},
	\end{cases}
	$$
	when $t\geq R_1$.
	
	The next result highlights the relationship between the critical values of the functionals $I_F$ and $\overline{I}_F$.
	\begin{lemma}\label{j-prop}
		The functional $\overline{I}_F$ given by \eqref{truncatedfunctional} satisfies the following properties: 
		\begin{itemize}
			\item $\overline{I}_F\in C^1(E_p)$;
			\item If $\overline{I}_F(u)\leq 0$ then $\Vert u \Vert<R_0$ and $I_F(v)=\overline{I}_F(v)$ for all $v$ in a small neighborhood of u;
			\item If $\mu$ is small enough, $\overline{I}_F$ satisfies the $(PS)_c$ condition with $c<0$.
		\end{itemize}
	\end{lemma}
	\begin{proof}
		Clearly, $ \overline{I}_F\in C^1(E_p)$. Now, if \( \|u\| > R_0 \) then \( \overline{I}_F(u) \geq \overline{h}(\|u\|) >0\), which shows the second assertion. Finally, observe that every Palais-Smale sequence for $\overline{I}_F$ at level \( c < 0 \) must be bounded. Then, if \( \mu \) is sufficiently small such that \[ \mathscr{C}_q  S^{\frac{pN}{2(p+1)}}-k\mu^\frac{s+1}{q-s}>0,\]
		we can invoke Proposition \ref{propcompacidade} to conclude that $\overline{I}_F$ satisfies the $(PS)_c$ condition at any level $c<0$.
	\end{proof}
	
	To construct a suitable sequence of negative critical values for \( \overline{I}_F \), we use the concept of genus of a set.
	
	\begin{definition}
		Let \( X\) be a Banach space and denote by	
		$$\mathcal{S}(X) := \{ A \subset X \setminus \{0\} : A \text{ is closed and if } u\in A \Longrightarrow  -u\in A \}.$$
		Let \( A \in \mathcal{S}(X) \). We say that \( A \) has genus \( n \) if \( n \) is the smallest natural number for which there exists a continuous odd function \( \varphi : A \to \mathbb{R}^n \setminus \{0\} \). The genus of \( A \) will be denoted by \( \gamma(A) \). If \( A = \emptyset \), we define \( \gamma(A) = 0 \). If \( A \neq \emptyset \) and there is no \( n \in \mathbb{N} \) such that there exist continuous odd functions \( f : A \to \mathbb{R}^n \setminus \{0\} \), we define \( \gamma(A) = \infty \).
	\end{definition}
	For every $A, B \in \mathcal{S}(X)$, we have \cite{Castro1980}:
	\begin{enumerate}[label = $\rm(G_{\arabic*})$]
		\item If there exists a continuous odd function \( \varphi : A \to B \), then \( \gamma(A) \leq \gamma(B) \);
		
		\item If \( A \subset B \), then \( \gamma(A) \leq \gamma(B) \);
		
		\item If \( \gamma(B) < \infty \), then \( \gamma(\overline{A - B}) \geq \gamma(A) - \gamma(B) \);
		
		\item If \( A \) is compact, then \( \gamma(A) < \infty \), and there exists \( \delta > 0 \) such that \( \gamma(A) = \gamma(\overline{N_{\delta}(A)}) \), where \( N_{\delta}(A) = \{ x \in X ,\, d(x, A) < \delta \} \);	
		\item $\gamma(A)\leq \#A$, where $\#A$ denotes the cardinality of $A$;
		\item If \( S^{N-1} \) is the sphere in \( \mathbb{R}^N \), then \( \gamma(S^{N-1}) = N. \)
	\end{enumerate}
	
	For every $a>0$, we adopt the notation
	\[ \mathcal{I}^{a} := \{ u \in E_p,\,\,\overline{I}_F(u) \leq a \} .\]
	\begin{lemma}\label{genus-finite-subspace}
		Given \( n \in \mathbb{N} \), there exists \( \varepsilon = \varepsilon(n) > 0 \) such that
		\[
		\gamma\left(\mathcal{I}^{-\varepsilon}\right) \geq n.
		\] 
	\end{lemma}
	\begin{proof}
		For each \( n \in \mathbb{N}\), let \( E_{p,n} \) be an \( n \)-dimensional subspace of \( E_p \). Consider \( u_n \in E_{p,n} \) with \( \lVert u_n \rVert = 1 \). For \( 0 < \rho < R_0 \), we have
		\[
		\overline{I}_F(\rho u_n) = I_F(\rho u_n) = \displaystyle \int_{\Omega} \left[F(x,\rho\Delta u_n)-\frac{\mu\rho^{s+1}}{s+1} \displaystyle g|u_n|^{s+1} -\frac{\rho^{q+1}}{q+1} |u_n|^{q+1}\phi(u_n)\right]dx.
		\]	
		Setting,
		\[
		\alpha_n = \inf\left\{\int_\Omega |u_n|^{q+1} \, dx \ :\ u_n \in E_{p,n}, \lVert u_n \rVert = 1\right\} > 0,
		\]
		
		\[
		\beta_n = \inf\left\{\int_\Omega g|u_n|^{s+1} \, dx \ :\ u_n \in E_{p,n}, \lVert u_n \rVert = 1\right\} > 0.
		\]
		we have
		\[
		\overline{I}_F(\rho u_n) \leq \frac{p}{p+1}\rho^{\frac{p+1}{p}} - \frac{\mu\beta_n}{s+1}\rho^{s+1} - \frac{\alpha_n}{q+1}\rho^{q+1}.
		\]
		At this moment, we pick \( \varepsilon=\varepsilon(n) >0\)  and \( \rho_0< R_0 \) such that \( \overline{I}_F(\rho_0 u_n) \leq -\varepsilon \) for \( u_n \in E_{p,n} \) with \( \lVert u_n \rVert = 1 \). Then,
		\[
		\{ u \in E_p :\, \lVert u \rVert = \rho_0 \}  \cap E_{p,n} \subset \mathcal{I}^{-\varepsilon}.
		\]
		Finally,  (G$_{2}$) and (G$_6$) give us
		\[
		\gamma\left(\mathcal{I}^{-\varepsilon}\right) \geq \gamma(\{ u \in E_p  :\, \lVert u \rVert = \rho_0 \} \cap E_{p,n}) = n,
		\qedhere
		\]
	\end{proof}
	
	Now we are in a position to show our main result.
	\subsection{Proof of Theorem \ref{theo1}}
	Let \( \Sigma_l := \{ Y \subset E_p \setminus \{ 0 \},\ Y = \overline{Y} = -Y, \ \gamma(Y) \geq l \} \),
	$$ c_l := \inf_{Y \in \Sigma_l} \sup_{u \in Y} \overline{I}_F(u) \quad \text{and} \quad K_c := \{ u \in E_p : \overline{I}_F^\prime(u) = 0, \overline{I}_F(u) = c \}. $$
	
	If there are infinitely many distinct $c_l$, the result is proved. Otherwise,  infinitely many $c_l$ are equal. Let $\mu$ be suitable small satisfying the conditions of Lemma \ref{j-prop}. by Lemma \ref{genus-finite-subspace}, for every \( l \in \mathbb{N} \), there exists \( \varepsilon(l) > 0 \) such that \( \gamma(\mathcal{I}^{-\varepsilon}) \geq l \). Since \( \overline{I}_F \) is even and continuous, \( \mathcal{I}^{-\varepsilon} \in \Sigma_l \) and \( c_l > -\infty \) for all \( l \).
	
	Now, we claim that, if \[ c = c_l = c_{l+1} = \dots = c_{l+m}, \quad m \in \mathbb{N},\]
	then \( \gamma(K_c) \geq m + 1 \). 
	Indeed, suppose by contradiction that  \( \gamma(K_c) \leq m \). By the compactness of \( K_c \) and the fact that \( \overline{I}_F \) satisfies the $(PS)$ condition on \( K_c \) for \( c < 0 \), the property $(G_4)$ guarantees the existence of a closed and symmetric set \( U \) such that \( K_c \subset U \) and \( \gamma(U) \leq m \). By the Clark's Deformation Lemma \cite{Castro1980, R2}, there exists an odd homeomorphism
	\[
	\eta : E_p \to E_p
	\]
	such that \( \eta(\mathcal{I}^{c+\delta} - U) \subset \mathcal{I}^{c-\delta} \) for some \( -c>\delta > 0 \). By definition of $c_{l+m}$, there exists \( A \in \Sigma_{l+m} \) such that
	\[
	\sup_{u \in A} \overline{I}_F(u) < c + \delta,
	\]
	which implies that \( A \subset \mathcal{I}^{c+\delta} \). Consequently,
	\begin{equation}\label{etaA-u}
		\eta(A - U) \subset \eta(\mathcal{I}^{c+\delta}) \subset \mathcal{I}^{c-\delta}.
	\end{equation}
	
	However, by $(G_1)$ and $(G_3)$ we get 
	\[ \gamma(\eta(\overline{A - U}))\geq  \gamma(\overline{A - U}) \geq \gamma(A) - \gamma(U) \geq l,\] 
	Therefore, \( \eta(\overline{A - U}) \in \Sigma_l \). As a byproduct, we have
	\[
	\sup_{u \in \eta(\overline{A - U})} \overline{I}_F(u) \geq c_l = c
	\]
	which contradicts \eqref{etaA-u}. Hence, $\gamma(K_c) \geq m + 1$ and by (G$_5$), there exist infinitely many distinct critical points associated with these critical values.
	
	This completes the proof. 
	
	\section{The superlinear case: existence of weak solutions}\label{sec-sup}
	
	
	To establish the main results of this section, we need some additional assumptions.
	
	\begin{enumerate}
		\item[$(f_5)$] {\it Additional dimensional relations:} We assume that 
		$$s+1>\begin{cases}
			q-p, & \text{if}\,\,\,\,  \dfrac{2}{N-2} < p < p_1(N)\\[10pt]
			\dfrac{(p-1)(q+1)}{p}, & \text{if}\,\,\,\,  p \geq p_2(N),
		\end{cases}$$
		where
		$$p_1(N):=\begin{cases}
			\dfrac{N}{N-2}, & \text{if}\,\,\,\,  3 \leq N \leq 5,\\[10pt]
			\dfrac{2+\sqrt{2N}}{N-2}, & \text{if}\,\,\,\,  N > 5
		\end{cases}\quad
		\text{and}
		\quad
		p_2(N):=\begin{cases}
			\dfrac{N}{N-2}, & \text{if}\,\,\,\,  3 \leq N \leq 5,\\[10pt]
			\dfrac{N-2}{2}, & \text{if}\,\,\,\,  N > 5.
		\end{cases}$$
	\end{enumerate} 
	
	Our first main result in the superlinear case reads as follows.
	\begin{theorem}\label{theo2}
		Suppose that ($f_1$)-($f_5$) hold and 
		$$\frac{1}{r}<s<q.$$
		Then, \eqref{prob}-\eqref{navi} and \eqref{prob}-\eqref{Diri} have a weak solution.
		
		If in addition,
		\begin{equation}\label{mulambda}
			\mu<(s+1)(2|\Omega|)^{\frac{r-p}{r(p+1)}}\frac{C_r}{C_{g,s}},
		\end{equation}
		then the same result is valid for $s=\frac{1}{r}$.
	\end{theorem}
	
	Under some additional assumptions on $f$, the restrictions in $(f_5)$ can be relaxed. For the next result, we consider:
	\begin{enumerate}
		\item[$(f_6)$]{\it Upper bound for primitive:} There exist $c_{p,r}>0$ and $t_0>0$ such that 
		$$F(x,t)\leq \frac{p}{p+1}f(x,t)t\leq \frac{p}{p+1}\left[t^{\frac{p+1}{p}}-c_{p,r}t^{\frac{r+1}{p}}\right],\quad t\geq t_0, \,x\in \Omega.$$
		\item[$(f_7)$] {\it Additional dimensional relations II:} We assume that $p>\dfrac{2}{N-2}$ for every $N\geq 3$ and 
		$$s+1>\dfrac{(p-1)(q+1)}{p} \,\,\,\text{for}\,\,\, p>\begin{cases}
			\dfrac{7}{2}, & \text{if}\,\,\,\,  N=3\\[10pt]
			\dfrac{N+2}{N-2}, & \text{if}\,\,\,\,  3<N\leq 6.
		\end{cases}$$
	\end{enumerate}

	\begin{theorem}\label{theo3}
		Under assumptions ($f_1$)-($f_4$), ($f_6$) and ($f_7$), the conclusion of the Theorem \ref{theo2} is valid.
	\end{theorem}

	To prove Theorems \ref{theo2} and \ref{theo3}, we appeal to the classical Mountain Pass Theorem \cite{Ambrosetti-Rabinowitz}.
	
	\begin{theorem}[Mountain Pass]\label{MP}
		Let $X$ be a Banach space and $I \in C^{1}\left(X, \mathbb{R}\right)$. If
		\begin{itemize}
			\item  there exists $v \in X$ and $d>0$ such that $\|v\|>d$ and
			\begin{equation}\tag{Geometric Condition}
				\inf _{\|u\|=d} I(u)>I(0) \geq I(v); 
			\end{equation}
			\item $I$ satisfies the $(PS)_{c}$ condition at the level
			\begin{equation}\tag{Mountain Pass Level}
				c:=\inf _{\gamma \in \Gamma} \max _{t \in[0,1]} I(\gamma(t)),   
			\end{equation}
			where
			$$\Gamma:=\{\gamma \in C([0,1], X) ; \gamma(0)=0, I(\gamma(1))<0\},$$
		\end{itemize}
		then $c$ is a critical value of $I$.
	\end{theorem}

	\subsection{The Geometric Condition}
	
	\begin{prop}\label{prop-mpgeometry}
		Suppose that ($f_1$)-($f_4$) hold and 
		$$\frac{1}{r}<s<q.$$
		Then, $I_F$ satisfies the Geometric Condition of the Mountain Pass Theorem. If in addition \eqref{mulambda} holds, the Geometric Condition is also valid for $s=1/r$.
	\end{prop}
	
	\begin{proof}
		Observe that $I_F(0) = 0$ and, by $(f_1)$,  
		$$
		I_F(u)\leq \frac{p}{p+1} \Vert u\Vert ^{\frac{p+1}{p}}
		-\frac{\mu}{s+1} |u|_{s+1}^{s+1}
		-\frac{1}{q+1} |u|_{q+1}^{q+1},
		\quad \forall \, \,u \in E_p.
		$$
		Noting that $q+1>\frac{p+1}{p}$, we have $I_F(tu)\rightarrow -\infty$ as $t\rightarrow \infty$ for every $u\in E_p\backslash\{0\}$. Then, fixing $u_0\in E_p\backslash\{0\}$, there exists $t_0>0$ such that $I_F(v)<0$, where we have set $v:=t_0 u_0$.  
		
		Now, we set
		$$\tilde{d}:=\left(\frac{C_p}{C_r}\right)^\frac{pr}{p-r}|\Omega|^{\frac{p}{p+1}}>0$$
		and take $u \in E_p$ such that $\|u\|\leq \tilde{d}$. From Lemma \ref{prop-fun-h}, we deduce
		\begin{equation*}
			I_F(u)\geq (2|\Omega|)^{\frac{r-p}{r(p+1)}}C_r\|u\|^{\frac{r+1}{r}}
			-\frac{C_{g,s}\mu}{s+1} \|u\|^{s+1}
			-\frac{S^{-1}}{q+1}\|u\|^{q+1}.        
		\end{equation*}
		When $\tfrac{1}{r}<s<q$, there exists $d\leq \tilde{d}$ such that $\displaystyle\inf_{\|u\|=d}I_F(u)>0$.
		
		On the other hand, assuming that $\tfrac{1}{r}=s<q$, then 
		\begin{equation*}
			I_F(u)\geq \left((2|\Omega|)^{\frac{r-p}{r(p+1)}}C_r-\frac{C_{g,s}\mu}{s+1} \right)\|u\|^{\frac{r+1}{r}}
			-\frac{S^{-1}}{q+1}\|u\|^{q+1}.        
		\end{equation*}
		By using the fact that \eqref{mulambda}, we arrive at $\displaystyle\inf_{\|u\|=d}I_F(u)>0$ for some $d \leq \tilde{d}$.  
		
		The proof is complete.
	\end{proof}

	\subsection{Upper bound for the Mountain Pass Level}\label{sectionMP}
	
	Assume that the assumptions of Theorem \ref{theo2} hold. Observing that 
	$$I_F(u)\leq \Phi(u):= \int_\Omega |u|^\frac{p+1}{p}dx-\frac{1}{q+1}|u|^{q+1}_{q+1} - \frac{\mu}{s+1} |u|^{s+1}_{s+1}$$
	and invoking \cite[Lemma 2.6]{EdJe-djairo}, we conclude that the Mountain Pass Level $c$ of $I_F$ satisfies 
	\begin{equation}\label{level}
		c \in \left(0,\frac{2}{N} S^{\frac{pN}{2(p+1)}}\right).    
	\end{equation}
	
	In order to find an upper bound for the Mountain Pass Level under the assumptions of Theorem \ref{theo3}, we prepare some preliminary facts. Let $\varphi$ be a positive radial solution of the problem \cite{HulshofVan96}
	\begin{equation*}
		\Delta (|\Delta\varphi|^{\frac{1}{p}-1}\Delta\varphi) = |\varphi|^{q-1} \varphi \ \ \text{in} \ \ \mathbb{R}^N.
	\end{equation*}
	Given $a \in \Omega$, let $\xi_a \in C_c^{\infty} (\mathbb{R}^N)$ be a function such that 
	$$0 \leq \xi_a(x)\leq 1,  \,\,x \in \mathbb{R}^N, \qquad \xi_a \equiv 1\,\,\text{in}\,\, B(a, \rho/2), \qquad\xi_a \equiv 0 \,\,\text {in} \,\, \mathbb{R}^{N}\backslash B(a, \rho)$$
	where $\rho>0$ and $B(a, \rho):=\{x \in \Omega,\,\,|x-a|<\rho\}$.  We define
	$$U_{\delta,a}:=\delta^{\frac{-N}{q+1}}\xi_a(x)  \varphi\left(\frac{x-a}{\delta}\right)\,\, \text{ and }\,\, V_{\delta,a}:=|U_{\delta,a}|^{-1}_{q+1}U_{\delta,a}.$$
	From \eqref{rssuperlinear} and ($f_1$), we deduce
	$$
	\lim_{t\to\infty}I_F(tV_{\delta,a})=-\infty,
	$$
	which implies that $\displaystyle\max_{t\geq0}I(tV_{\delta,a})$ is attained at some $t_\delta>0$. Then,
	$$
	0=I_F'(t_\delta V_{\delta,a})=\displaystyle\int_\Omega f(x,t_\delta \Delta V_{\delta,a})\Delta V_{\delta,a}\, dx-t_\delta^s|V_{\delta,a}|^{s+1}_{s+1}-t_\delta^q,
	$$
	from where we infer that
	\begin{equation}\label{I'tv}
		t_\delta^{q+1}=\displaystyle\int_\Omega f(x,t_\delta \Delta V_{\delta,a})t_\delta\Delta V_{\delta,a}\, dx-t_\delta^{s+1}|V_{\delta,a}|^{s+1}_{s+1}.
	\end{equation}
	
	\begin{lemma}\label{tdeltaltd}
		There exist $0<b<\overline{b}<+\infty$ and $\delta_0 \in (0,1)$ such that $t_\delta\in [b,\overline{b}]$, for all $\delta \in (0,\delta_0)$. 
	\end{lemma}
	\begin{proof}
		Let $\delta \in (0,1)$. Using \eqref{I'tv} and applying the same strategy from \cite{EdJe-djairo}, we can find a constant $C_0>0$ such that
		$$
		t_\delta^{q} \leq t_\delta^\frac{1}{p} S+t_\delta^\frac{1}{p}C_0\delta \Longrightarrow t_\delta^\frac{pq-1}{p} \leq  S+C_0 \Longrightarrow t_\delta \leq  \left(S+C_0\right)^\frac{p}{pq-1}=:\overline{b}.
		$$
		
		On the other hand,  we set
		$$P_\delta:=\{x\in \Omega;|t_\delta \Delta V_{\delta,a}(x)|<t_0 \}, \qquad Q_\delta=\Omega\backslash P_\delta.$$ 
		For $\delta$ small enough, we have
		\begin{eqnarray*}
			C_1t_\delta^{1/r}  &\leq &  \displaystyle c_p t_\delta^{1/p} \int_{Q_\delta}  |\Delta V_{\delta,a} (x)|^\frac{p+1}{p}dx + c_r t_\delta^{1/r} \int_{P_\delta} |\Delta V_{\delta,a} (x)|^\frac{r+1}{r}dx\\
			&\leq &  \int_\Omega  f(x,t_\delta \Delta V_{\delta,a} (x))\Delta V_{\delta,a} (x)dx\\
			&\leq&C_2\delta t_\delta^s +C_3 t_\delta^q.  
		\end{eqnarray*}
		for some positive constants $C_1$, $C_2$ and $C_3$. Hence,
		\[C_1\leq C_2  t_\delta^{s-\frac{1}{r}}+C_3 t_\delta^{q-\frac{1}{r}},\]
		which implies in the existence of $b>0$ such that $t_{\delta}>b$ for $\delta$ sufficiently small.
	\end{proof}
	
	For the next result, we consider the number 
	$$p_*(N):=\begin{cases}
		\dfrac{7}{2}, & \text{if}\,\,\,\,  N = 3,\\[10pt]
		\dfrac{N+2}{N-2}, & \text{if}\,\,\,\,  N > 3
	\end{cases}$$
	and the function
	$$f_*(t):=\begin{cases}
		f_1(t), & \text{if}\,\,\,\,  p \leq  p_*(N),\\[10pt]
		f_2(t), & \text{if}\,\,\,\,   p_*(N)<p\leq \frac{N^2+2N-4}{N^2-4N+4},\\[10pt]
		f_3(t), & \text{if}\,\,\,\,  p> \frac{N^2+2N-4}{N^2-4N+4},
	\end{cases}$$
	where $f_1(t):=t^\frac{p+1}{p}S$,
	
	$$f_2(t):=\left\{ \begin{array}{lll}
		t^\frac{p+1}{p}S+c_1t^\frac{r+1}{r} \delta^{\frac{N(r+1)}{r(p+1)}}-c_2 t^\frac{r+1}{p}\delta^{\frac{N}{q+1}\frac{N}{N-2}}, &\text{if }\, r<\frac{2}{N-2},\\
		t^\frac{p+1}{p}S+c_1t^\frac{r+1}{r} \delta^{\frac{N(r+1)}{r(p+1)}}-c_2\delta^{\frac{N}{q+1}\frac{N}{N-2}}|\log(\delta)|, &\text{if }\, r=\frac{2}{N-2},\\
		t^\frac{p+1}{p}S+c_1t^\frac{r+1}{r} \delta^{\frac{N(r+1)}{r(p+1)}}-c_2 t^\frac{r+1}{p}\delta^{\frac{N(p-r)}{p+1}}, &\text{if }\, r>\frac{2}{N-2},
	\end{array}
	\right. $$
	and
	$$f_3(t):= \left\{ \begin{array}{lll}
		t^\frac{p+1}{p}S+c_1t^\frac{r+1}{r} \delta^{\frac{N(r+1)}{r(p+1)}}- c_2\lambda\delta^{\frac{Nq}{q+1}}|\log(\delta)|, & \! \text{if } r+1=\frac{p+1}{q+1},\\
		t^\frac{p+1}{p}S+c_1t^\frac{r+1}{r} \delta^{\frac{N(r+1)}{r(p+1)}}-c_2\lambda\delta^{\frac{Nq}{q+1}}, & \! \text{if } r+1<\frac{p+1}{q+1},\\
		t^\frac{p+1}{p}S+c_1t^\frac{r+1}{r} \delta^{\frac{N(r+1)}{r(p+1)}}-c_2\lambda\delta^{\frac{N(p-r)}{p+1}}, &  \! \text{if }  r+1>\frac{p+1}{q+1}.
	\end{array}
	\right.$$
	Here, $c_1$ and $c_2$ are positive constants that will be determined below.
	\begin{lemma}\label{estimateFtv}
		Under assumptions of Theorem \ref{theo3}, there exist $c_1$, $c_2>0$, $0<m<\overline{m}$, independent of $\delta$, and $\delta_0 \in (0,1)$ such that
		$$\int_\Omega f(x,t\Delta V_{\delta,a}(x))(t\Delta V_{\delta,a}(x))dx<f_*(t), $$
		for every $t\in[m,\overline{m}]$ and $\delta\in(0,\delta_0)$.
	\end{lemma}
	\begin{proof}
		The result follows by combining the assumptions ($f_6$) and ($f_7$) with the arguments from \cite[Lemma 4.9]{guimaraes2023hamiltonian}.
	\end{proof}
	

	Now we are in position to state and prove the main result of this subsection.
	
	\begin{prop}\label{prop-mplevel}
		Under assumption of Theorem \ref{theo3}, the Mountain Pass Level $c$ of $I_F$ satisfies \eqref{level}.
	\end{prop}
	\begin{proof}
		From  $(f_6)$ and \eqref{I'tv}, we get
		\begin{eqnarray*}
			\max_{t\geq0} I_F(tV_{\delta,a})& =&  I_F(t_\delta V_{\delta,a})\\
			&=&  \int_\Omega F(x,t_\delta \Delta V_{\delta,a})dx-\frac{t_\delta^{q+1}}{q+1} - \frac{\mu t_\delta^{s+1}}{s+1} |V_{\delta,a}|^{s+1}_{s+1}\\
			&\leq&\int_\Omega  \frac{p}{p+1}f(x,t_\delta \Delta V_{\delta,a})t_\delta\Delta V_{\delta,a} dx -\frac{t_\delta^{q+1}}{q+1}- \frac{\mu t_\delta^{s+1}}{s+1}|V_{\delta,a}|^{s+1}_{s+1} \\
			&=& \frac{2}{N}\int_\Omega f(x,t_\delta \Delta V_{\delta,a})t_\delta\Delta V_{\delta,a} dx  - \frac{\mu(q-s)}{(q+1)(s+1)}t_\delta^{s+1}|V_{\delta,a}|^{s+1}_{s+1},
		\end{eqnarray*}
		Now, Lemma \ref{estimateFtv} and the expression \eqref{I'tv} allow us to repeat the arguments used in \cite{guimaraes2023hamiltonian} to deduce that
		$t_\delta< S^\frac{p}{pq-1}$. Hence,
		$$\max_{t\geq0} I_F(tV_{\delta,a})< \frac{2}{N}S^\frac{pN}{2(p+1)},
		$$
		which implies that $c$ satisfies \eqref{level}.
	\end{proof}
	
	\subsection{Proof of Theorems \ref{theo2} and \ref{theo3}: Conclusion}
	By Proposition \ref{propcompacidade} and Proposition \ref{prop-mpgeometry}, the functional $I_F$ satisfies all the hypothesis of Theorem \ref{MP}. Hence, $I_F$ has at least one critical point and consequently, \eqref{prob}-\eqref{navi} and \eqref{prob}-\eqref{Diri} have at least a weak solution.

	\section{Connections with Hamiltonian systems}\label{sec-connecHam}
	
	It is well-known that \eqref{prob}-\eqref{navi} has a natural connection with Hamiltonian systems of elliptic type, which have been extensively studied in the literature due to their rich variational structure and the qualitative properties of their solutions. In this section, we highlight how our results can impact the quantitative analysis of Hamiltonian systems in different scenarios.
	
	Let $f(x,t)=f(t)$ be an increasing function satisfying $(f_1)$--$(f_4)$, and let $u \in E_p$ be a weak solution of \eqref{prob}-\eqref{navi}. From \cite[Section 4]{ederson-JMAA}, the function $-v:=f(\Delta u)$ belongs to $W^{2, \frac{q+1}{q}}(\Omega) \cap W_0^{1, \frac{q+1}{q}}(\Omega)$, and the pair $(u,v)$ is a strong solution of the system
	\begin{equation*}
		\tag{HS}\label{sist}
		\left\{
		\begin{array}{lll}
			-\Delta u = f^{-1}(v) & \text{ in } \Omega,\\
			-\Delta v = \mu g(x)|u|^{s-1}u + |u|^{q-1}u & \text{ in } \Omega,\\
			u = v = 0 & \text{ on } \partial\Omega.
		\end{array}
		\right.
	\end{equation*}	
	This connection between \eqref{prob}-\eqref{navi} and \eqref{sist}, established through the transformation $v=-f(\Delta u)$, preserves the standard variational structure, which is crucial for applying our variational approach to improve the regularity of the solutions of \eqref{prob}-\eqref{navi}. To illustrate this claim, we consider a concrete realization of $f$ and discuss its corresponding Hamiltonian system.
	
	For every $\lambda>0$, we set
	\[
	\mathscr{F}_\lambda(t) := \lambda |t|^{r-1}t + |t|^{p-1}t.
	\]  
	From \cite{guimaraes2023hamiltonian}, the function \( f = \mathscr{F}_\lambda^{-1} \) satisfies the structural assumptions \((f_1)-(f_4)\) and by construction, any weak solution \( u \) of \eqref{prob}-\eqref{navi} corresponds to a strong solution \((u,v)\) of the following system:
	\begin{equation*}
		\label{sist1}
		\left\{
		\begin{array}{lll}
			-\Delta u = \lambda |v|^{r-1}v + |v|^{p-1}v & \text{in } \Omega,\\
			-\Delta v = \mu |u|^{s-1}u + |u|^{q-1}u & \text{in } \Omega,\\
			u = v = 0 & \text{on } \partial\Omega,
		\end{array}
		\right.
	\end{equation*}
	where \( \mu > 0 \) and the exponents \( p, q, r, s \) satisfy the conditions assumed in the present work. In this sense, our results not only complement the existing literature by addressing the case of \emph{sublinear unilateral perturbations} of the Lane-Emden critical system,
	$$\begin{cases}
		-\Delta u = |v|^{p-1}v & \text{in } \Omega,\\
		-\Delta v = |u|^{q-1}u & \text{in } \Omega,\\
		u = v = 0 & \text{on } \partial\Omega,
	\end{cases}$$
	but also extend the analysis developed in \cite{guimaraes2023hamiltonian} by establishing the existence of infinitely many weak solutions when at least one of the nonlinearities is sublinear. This contribution fills a natural gap in the study of such Hamiltonian systems, offering new insights into the multiplicity and qualitative behavior of solutions in the presence of lower-order perturbations. We remark, however, that the borderline case in which the exponents satisfy $r=\tfrac{1}{q}$ and $s=\tfrac{1}{p}$ remains open and poses additional challenges for future investigation.
	
	In the superlinear case, we can adopt the same strategy developed in \cite{guimaraes2023hamiltonian} to obtain weak/strong solutions of \eqref{sist1}  for every $N\geq 4$, or for $N=3$ and $ p\in \left(2,\frac{7}{2}\right]\cup \left(8,\infty\right)$.

	\vspace{3mm}
	
	\noindent {\bf Classical solutions.} When \( r > 1 \), we may adapt the ideas from \cite[Section~3]{hulshofvandervorst-93} to show that \( u, v \in L^\theta \) for all \( 1 \leq \theta < \infty \). By applying classical regularity theory for second-order elliptic equations, it follows that \( (u, v) \in C^{2,\sigma}(\overline{\Omega}) \times C^{2,\sigma}(\overline{\Omega}) \) for some \( \sigma \in (0,1) \) depending on \( p \) and \( q \). This additional regularity implies that the weak solutions of \eqref{prob}-\eqref{navi} is, in fact, a classical one. 
	
	When \( p \) and \( q \) lie below the critical hyperbola, that is,
	\[
	\frac{1}{p+1} + \frac{1}{q+1} > \frac{N-2}{N},
	\]
	the functional \( I_F \) satisfies the \((PS)\) condition on the whole space. This case was addressed in \cite{AgudeloRufVelez} for sublinear perturbations and several results concerning the existence and nonexistence of solutions were achieved. In this context, our variational approach remains applicable and ensures the existence of infinitely many solutions. Furthermore, by adapting the arguments from \cite[Section~4]{ederson-JMAA}, one can also prove that \( (u, v) \) is a classical solution of \eqref{sist1}. However, in the critical case, the standard bootstrap argument fails to establish the regularity \( u \in L^\theta\) for all \( \theta \geq 1 \).

	\vspace{3mm}
	
	\noindent {\bf Positivity and symmetry.} In such cases where classical regularity holds, one can follow the argument developed in \cite{guimaraes2023hamiltonian} to show that the solution associated with the critical value $c_1$ is, in fact, positive. Additionally, if $\Omega = B_r(0)$ is a ball centered at the origin, it is possible to use the arguments from \cite{EdJe-djairo} to prove that the corresponding solution is radial. These qualitative properties emphasize the significance of the first critical level and reinforce the structure and symmetry of the solutions obtained through the variational method.

\subsection*{Funding} The work of the first and last authors are supported by NSF of China, Nos. W2533024 and 12171087, respectively.

\subsection*{Disclosure of interest}
The authors report there are no competing interests to declare.\\

	\bibliographystyle{siam}
	
	\bibliography{refs.bib}

\end{document}